\documentclass[11pt]{amsart}

\usepackage[T1]{fontenc}
\usepackage[utf8]{inputenc}
\usepackage{lmodern}
\usepackage[expansion=false]{microtype}
\usepackage{amsmath,amssymb}
\usepackage{enumitem}
\usepackage[hidelinks]{hyperref}
\usepackage[nameinlink,capitalise,noabbrev]{cleveref}

\numberwithin{equation}{section}
\newtheorem{theorem}{Theorem}[section]
\newtheorem{proposition}[theorem]{Proposition}
\newtheorem{lemma}[theorem]{Lemma}
\newtheorem{corollary}[theorem]{Corollary}
\theoremstyle{definition}
\newtheorem{definition}[theorem]{Definition}
\newtheorem{convention}[theorem]{Convention}
\theoremstyle{remark}
\newtheorem{remark}[theorem]{Remark}

\DeclareMathOperator{\Pic}{Pic}
\DeclareMathOperator{\End}{End}
\DeclareMathOperator{\Hom}{Hom}
\DeclareMathOperator{\Kar}{Kar}
\DeclareMathOperator{\Add}{Add}
\DeclareMathOperator{\thick}{thick}
\DeclareMathOperator{\Spec}{Spec}
\DeclareMathOperator{\SH}{SH}
\DeclareMathOperator{\GW}{GW}
\DeclareMathOperator{\rk}{rk}
\newcommand{\MBP}{\mathrm{MBP}}

\newcommand{\F}{\mathbb F_2}
\newcommand{\X}{\mathbb X}
\newcommand{\one}{\mathbf 1}
\newcommand{\Cinf}{\mathcal C_{\infty}}
\newcommand{\Dinf}{\mathcal D_{\infty}}
\newcommand{\C}{\mathcal C}
\newcommand{\D}{\mathcal D}
\newcommand{\Hh}{\mathcal H}
\newcommand{\V}{\mathcal V}
\newcommand{\B}{\mathcal B}
\newcommand{\cell}{\mathrm{cell}}
\newcommand{\Mod}{\mathrm{Mod}}
\newcommand{\Sm}{\mathrm{Sm}}
\newcommand{\id}{\mathrm{id}}
\title[Picard groups and composition nilpotence]
{Picard groups and composition nilpotence for finite cellular isotropic
spectra}

\author{David Kumallagov}
\subjclass[2020]{Primary 14F42; Secondary 55P42, 18G80}
\keywords{isotropic motivic homotopy theory, finite cellular spectra,
motivic Brown--Peterson spectrum, Picard group, weight structure}
\date{}

\hypersetup{
  pdftitle={Picard groups and composition nilpotence for finite cellular
  isotropic spectra},
  pdfauthor={David Kumallagov},
  pdfkeywords={isotropic motivic homotopy theory, finite cellular spectra,
  motivic Brown--Peterson spectrum, Picard group, weight structure,
  composition nilpotence}
}

\begin{document}

\begin{abstract}
Let $k=k_0(t_1,t_2,\ldots)$ be a flexible field of characteristic different
from $2$, let $\X$ be the mod-$2$ isotropic sphere, and set
$E=\X\wedge\MBP$.  We prove
\[
  \Pic\bigl(\SH(k/k)^c_{\cell}\bigr)\cong\mathbb Z^2;
\]
thus every tensor-invertible finite cellular isotropic spectrum is a unique
bigraded suspension of $\X$.  More generally, $E_{**}$ is conservative on
finite cellular objects, and concentration on one diagonal forces a finite
direct sum of suspended isotropic spheres.  A bounded diagonal weight
structure recovers the exact weights and minimal-complex terms from
$E_{**}$.  For every nonzero finite cellular $M$, the kernel of
\[
  \End(M)\longrightarrow\End(E\wedge_{\X}M)
\]
is a composition-nilpotent ideal, with exponent at most
$d(M)(2L(M)-1)$, where $L(M)$ is the diagonal width and $d(M)$ the maximal
number of distinct Tate degrees on one diagonal.  This yields detection of
composition nilpotence and canonical Fitting decompositions.  
\end{abstract}

\maketitle

\section{Introduction and main results}\label{sec:introduction}

Isotropic localization was introduced by Vishik for motives \cite{VishikIsotropic} and developed by
Tanania in stable motivic homotopy theory
\cite{TananiaStable,TananiaCellular}.  It annihilates
varieties which doesn’t have closed point of degree prime to a given prime $p.$ Over a flexible field its
Brown--Peterson linear cellular category is exceptionally simple: for
\[
  E=\X\wedge\MBP,
\]
Tanania identifies cellular $E$-modules with bigraded $\F$-vector spaces.
What this calculation controls in the original $\X$-linear category is much
less formal.  Before $H\mathbb Z/2$-completion, neither conservativity of
extension of scalars nor the lifting of invertible objects and endomorphisms
follows from the split $E$-linear description.

We show that, on finite cellular objects, $E$-homology nevertheless retains
exactly the information needed to control invertibility and weights.  In
particular, every tensor-invertible finite cellular isotropic spectrum is a
unique bigraded suspension of the isotropic sphere, so the full Picard group
is $\mathbb Z^2$.  We also prove an exact weight-profile theorem and a
quantitative composition-nilpotence theorem for the objectwise kernel of
$E$-homology.  The latter yields polynomial lifting, canonical Fitting
decompositions, and the Krull--Schmidt property.

The decisive input is the calculation
\[
  \pi_{0,0}(\X)\cong\F.
\]
Tanania observed that quadratic transfers imply $8=0$ in this ring
\cite[Remark~3.6]{TananiaCellular}, and we sharpen the transfer argument:  the resulting relations force
$2=0$; connectivity makes $\GW(k)\to\pi_{0,0}(\X)$ surjective, and the
structural map $\X\to E$ identifies the quotient with $\F$.  This calculation
rules out exotic idempotent summands in the weight heart and makes its
$E$-linear reduction full and conservative.

Let us recall some chronology of calculations of Picard groups and related things in motivic categories. 
Hu studied Picard groups in stable $\mathbb A^1$-homotopy theory
\cite{HuPicard}, Bachmann proved the invertibility of reduced affine-quadric
motives \cite{BachmannAffine}, and Vishik determined the subgroup of the
Picard group of geometric motives generated by them
\cite{VishikAffinePicard}.  Du and Vishik introduced Morava-isotropic stable
categories and used them to construct points of the Balmer spectrum
\cite{DuVishikBalmer}; Vishik's comparison of isotropic and numerical
equivalence gives further tensor-triangular context \cite{VishikNumerical}.
Recent work of Sparks studies tensor-nilpotence in $\F$-linear isotropic Tate
and Artin--Tate categories of Voevodsky motives, under a standing
characteristic zero hypothesis \cite{SparksIsotropicTT}.  Our result instead
concerns ordinary composition-nilpotence of an objectwise $E$-homological
kernel in the stable motivic homotopy category; neither
nilpotence statement formally implies the other.

The theory of weight structures fits very well into this picture: Bondarko and Tabuada proved a general $w$-Picard theorem when there exists a
full, additive, conservative symmetric monoidal functor from the heart of a
bounded monoidal weight structure to a local semisimple symmetric monoidal
abelian category \cite[Theorem~1.1]{BondarkoTabuadaPicard}.  The reduction
$\Phi\colon\Hh\to\V$ constructed below satisfies these hypotheses.  Their
theorem therefore gives an alternative proof
\[
  \Pic(\C)\cong\Pic(\Hh)\times\mathbb Z\cong\mathbb Z^2.
\]
Our direct argument additionally gives purity, exact weight
profiles, and the quantitative composition-nilpotence results.  To the best
of our knowledge, both the calculation of $\pi_{0,0}(\X)$ and
the explicit objectwise nilpotence bound are new and very interesting.

Fix a flexible field
\[
  k=k_0(t_1,t_2,\ldots),
  \quad \operatorname{char}(k_0)\ne2.
\]
Let $\X$ denote the isotropic sphere over $k$.  Under Tanania's equivalence
$\SH(k/k)\simeq\X\text{-}\Mod$, let
$\X\text{-}\Mod_{\cell}$ be the localizing subcategory generated by the
bigraded free modules $\Sigma^{p,q}\X$, and set
\[
  \Cinf:=(\X\text{-}\Mod_{\cell})^\omega,
  \quad
  \C:=h\Cinf.
\]
Thus $\C$ is the thick idempotent-complete triangulated subcategory generated
by the bigraded isotropic spheres.  We write
\[
  \C=\SH(k/k)^c_{\cell}.
\]
The compact objects form a stable idempotent-complete subcategory.  Because
$\wedge_{\X}$ is exact in each variable and sends pairs of cells to cells,
the thick closure of the cells is closed under $\wedge_{\X}$.  Hence $\Cinf$
and $\C$ are symmetric monoidal subcategories.  For any symmetric monoidal category
$\mathcal T$, its Picard group $\Pic(\mathcal T)$ is the abelian group of
isomorphism classes of objects $P$ admitting an object $P^{-1}$ with
$P\otimes P^{-1}\cong\one$; the group operation is induced by tensor
product.  In particular, $\Pic(\C)$ is formed using $\wedge_{\X}$.

Let $\MBP$ be the motivic Brown--Peterson spectrum at the prime $2$, in the
conventions of \cite{TananiaCellular,VezzosiMBP}, and put
\[
  E:=\X\wedge\MBP.
\]
The structural map $\X\to E$ induces an exact symmetric monoidal extension of
scalars functor
\[
  F:=E\wedge_{\X}-\colon\C\longrightarrow\D,
\]
where
\[
  \Dinf:=(E\text{-}\Mod_{\cell})^\omega,
  \qquad
  \D:=h\Dinf.
\]
All Hom and End groups below are taken in these homotopy categories unless a
different category is displayed explicitly.

We use $[1]=\Sigma^{1,0}$ for the triangulated suspension and set
\[
  T:=\Sigma^{1,1}\X,
  \quad
  U:=\Sigma^{1,1}E.
\]
Here $T^q$ and $U^q$ denote tensor powers for $q\geq0$ and tensor powers of
the respective inverse objects for $q<0$; the integer $q$ is called the
\emph{Tate degree}.  
Consequently,
\begin{equation}\label{eq:suspension-convention}
  \Sigma^{p,q}\X=T^q[p-q],
  \qquad
  \Sigma^{p,q}E=U^q[p-q].
\end{equation}
For $M\in\C$, its isotropic $\MBP$-homology is
\[
  E_{**}(M):=\pi_{**}(E\wedge_{\X}M).
\]
We call $p-q$ the \emph{diagonal degree} of bidegree $(p,q)$.

Our argument combines the split $E$-linear calculation with a bounded
diagonal weight structure.  The calculation $\End_{\C}(\X)\cong\F$ makes
reduction to $E$ full and conservative on the weight heart; weight detection
then lifts purity and conservativity to all finite cellular objects.

\subsection*{Main results}

All results below concern finite cellular objects over a flexible field of
characteristic different from $2$.  All endomorphisms have bidegree $(0,0)$,
and nilpotence means nilpotence under ordinary composition. 

Our principal result is the following.

\begin{theorem}\label{thm:main-picard}
The canonical homomorphism
\[
  \mathbb Z^2\longrightarrow\Pic(\C),
  \qquad
  (p,q)\longmapsto[\Sigma^{p,q}\X],
\]
is an isomorphism.
\end{theorem}

The classification of invertible objects follows from a more general finite
cellular detection statement.

\begin{theorem}[Pure finite cellular $\MBP$--Nakayama theorem]
\label{thm:pure-nakayama}
Let $M\in\C$.
\begin{enumerate}[label=\textup{(\roman*)}]
\item If $E_{**}(M)=0$, then $M\simeq0$.
\item Suppose that $E_{**}(M)$ is concentrated in one diagonal degree $n$.
For
\[
  m_q:=\dim_{\F}E_{n+q,q}(M),
\]
only finitely many $m_q$ are nonzero, and there is an isomorphism
\[
  M\simeq
  \bigoplus_{q\in\mathbb Z}
  \bigl(\Sigma^{n+q,q}\X\bigr)^{\oplus m_q}.
\]
\end{enumerate}
\end{theorem}

\begin{corollary}[One-cell $\MBP$--Nakayama theorem]
\label{cor:one-cell-nakayama}
If $M\in\C$ and
\[
  E_{**}(M)\cong\Sigma^{p,q}\F,
\]
then
\[
  M\simeq\Sigma^{p,q}\X.
\]
\end{corollary}

\begin{corollary}
\label{cor:equivalence-detection}
A morphism $f\colon M\to N$ in $\C$ is an isomorphism if and only if
$E_{**}(f)$ is an isomorphism.
\end{corollary}

The same structure controls weights and maps, not only objects.  For a
nonzero $M\in\C$, put
\[
  \beta_{n,q}(M):=\dim_{\F}E_{n+q,q}(M),
\]
and write
\[
  \dim_{\F}E_{**}(M):=\sum_{p,q}\dim_{\F}E_{p,q}(M)
\]
for the total dimension over all bidegrees.  This sum is finite for
$M\in\C$.  Define
\begin{equation}\label{eq:homological-width-data}
\begin{aligned}
  a(M)&:=\min\{n\mid \beta_{n,q}(M)\ne0\text{ for some }q\},\\
  b(M)&:=\max\{n\mid \beta_{n,q}(M)\ne0\text{ for some }q\}.
\end{aligned}
\end{equation}
\[
  L(M):=b(M)-a(M)+1,
  \quad
  d(M):=\max_n\#\{q\mid\beta_{n,q}(M)\ne0\}.
\]
These are finite positive integers except that $a(M)$ and $b(M)$ are not
defined for $M=0$.

%For a two-sided ideal $I$ in a ring, $I^r$ denotes the additive subgroup
%generated by all products of $r$ elements of $I$; in particular, $I^r=0$
%means that every such product vanishes.

Let $w$ denote the bounded ''diagonal'' weight structure on $\C$, constructed in
\cref{sec:weights} and characterized by placing every $T^q$ in weight zero.

\begin{theorem}\label{thm:weight-profile}
Let $M\in\C$ be nonzero.  For integers $a\le b$, one has
\[
  M\in\C_{w\in[a,b]}
  \quad\Longleftrightarrow\quad
  E_{p,q}(M)=0\text{ whenever }p-q\notin[a,b].
\]
Thus $[a(M),b(M)]$ is the smallest weight interval containing $M$.
Moreover, if $C^\bullet$ is any $\Phi$-minimal bounded complex representing
the strong weight complex of $M$, then
\[
  C^{-n}\cong
  \bigoplus_{q\in\mathbb Z}(T^q)^{\oplus\beta_{n,q}(M)}
  \qquad(n\in\mathbb Z).
\]
In particular, its termwise multiplicities and its cohomological support are
intrinsic and are read directly from $E_{**}(M)$.
\end{theorem}

\begin{theorem}
\label{thm:nilpotence}
For $M\in\C$, let
\[
  I_M:=\ker\!\left(
  \End_{\C}(M)\longrightarrow
  \End_{\D}(E\wedge_{\X}M)
  \right).
\]
Then $I_M$ is a nilpotent two-sided ideal.  In particular, every 
endomorphism $f\colon M\to M$ with $E_{**}(f)=0$ is nilpotent.

More explicitly, if $M\ne0$, then the invariants in
\eqref{eq:homological-width-data} give the  bound
\[
  I_M^{\,d(M)(2L(M)-1)}=0.
\]
The exponent depends on the number of \emph{distinct} Tate degrees in each
diagonal degree, not on their multiplicities.  
\end{theorem}

\begin{corollary}
\label{cor:nilpotence-detection}
Let $M\in\C$ be nonzero and set
\[
  N(M):=d(M)(2L(M)-1).
\]
\begin{enumerate}[label=\textup{(\roman*)}]
\item An endomorphism $f$ of $M$ is nilpotent if and only if
$E_{**}(f)$ is nilpotent.  If $E_{**}(f)^m=0$, then $f^{mN(M)}=0$.
\item More generally, if $\mathcal J\subseteq\End_{\C}(M)$ is a two-sided
ideal and every product of $m$ elements of its image in
$\End_{\D}(FM)$ is zero, then $\mathcal J^{mN(M)}=0$.
\item If $P(x)\in\F[x]$ and $P(E_{**}(f))=0$, then
\[
  P(f)^{N(M)}=0.
\]
In particular, if $\mu_{\bar f}$ is the minimal polynomial of
$\bar f=E_{**}(f)$, then $\mu_{\bar f}(f)^{N(M)}=0$.  Thus $f$ is algebraic
over $\F$ and has an annihilating polynomial of degree at most
\(N(M)\dim_{\F}E_{**}(M).
\)
\end{enumerate}
\end{corollary}

\begin{theorem}[Motivic Fitting decomposition]
\label{thm:fitting}
For every $M\in\C$ and $f\in\End_{\C}(M)$ there is a unique idempotent
$e_f\in\End_{\C}(M)$ commuting with $f$ such that the induced decomposition
\[
  M\cong M_{\mathrm{nil}}\oplus M_{\mathrm{inv}},
  \qquad
  M_{\mathrm{nil}}=\operatorname{im}(e_f),
  \quad
  M_{\mathrm{inv}}=\operatorname{im}(1-e_f),
\]
has $f|_{M_{\mathrm{nil}}}$ nilpotent and
$f|_{M_{\mathrm{inv}}}$ is invertible.  After applying $E_{**}$ this is the
''ordinary'' Fitting decomposition of the finite-dimensional bigraded
$\F$-vector space $E_{**}(M)$.  The projectors are functorial for
intertwining morphisms: if $uf=f'u$, then $ue_f=e_{f'}u$.
\end{theorem}

The proof has four principal ingredients.  First, stable
$\mathbb A^1$-connectivity makes the family $\{T^q\}_{q\in\mathbb Z}$
negative, and hence produces a bounded weight structure.  Second, quadratic
transfers give $\End_{\C}(\X)=\F$.  Third, this calculation rules out exotic
retracts in the weight heart and makes extension of scalars full and
conservative on that heart.  Bondarko's weight-detection theorem then proves
\cref{thm:pure-nakayama}.  Finally, nilpotence follows by combining a finite
matrix argument on a minimal strong weight complex with a finite-filtration
''ghost argument'' for the kernel of the strong weight-complex functor.

\subsection*{Outline}

The article is organized as follows.  In \cref{sec:isotropic-input} we fix
conventions, recall the isotropic construction, and state the $E$-linear
input.  Connectivity is proved in \cref{sec:connectivity}.
\Cref{sec:weights} recalls weight structures and constructs the diagonal
weight structures.  The endomorphism calculation appears in
\cref{sec:end-unit}.  The heart and minimal complexes are studied in
\cref{sec:heart,sec:minimal-complexes}.  We prove the Nakayama and Picard
theorems in \cref{sec:nakayama,sec:picard}.  The exact ''weight profile'',
nilpotence, polynomial lifting, Fitting decomposition, and Krull--Schmidt consequences are
proved in \cref{sec:nilpotence}.

\section{Isotropic and cellular input}\label{sec:isotropic-input}

\subsection{Bigrading and finite cellular objects}

For a motivic spectrum $Y$, we use the convention
\[
  \pi_{p,q}(Y):=[\Sigma^{p,q}\one,Y]_{\SH(k)}.
\]
If $R$ is a motivic $E_\infty$-ring, then
$R\text{-}\Mod_{\cell}$ denotes the localizing subcategory of $R$-modules
generated by the free cells $\Sigma^{p,q}R$; see
\cite{DuggerIsaksen,TananiaCellular}.

\begin{convention}\label{conv:finite-cell}
An object of $R\text{-}\Mod_{\cell}$ is called \emph{finite cellular} if it belongs to
\[
  \thick\{\Sigma^{p,q}R\mid p,q\in\mathbb Z\},
\]
the smallest stable idempotent-complete subcategory containing the cells.

\end{convention}

Since the cells are compact and generate $R\text{-}\Mod_{\cell}$, the
standard compact-generation theorem gives
\[
  (R\text{-}\Mod_{\cell})^\omega
  =\thick\{\Sigma^{p,q}R\mid p,q\in\mathbb Z\},
\]
see \cite[Lemma~4.4.5]{NeemanTriangulated}.  We write
$\Add(\mathcal S)$ for
closure of a collection $\mathcal S$ under finite direct sums and isomorphisms,
and $\Kar(\mathcal S)$ for its retraction closure.

\subsection{Flexible fields and the isotropic sphere}

\begin{definition}\label{def:flexible}
A field is \emph{flexible} if it is isomorphic to a purely transcendental
extension
\[
  k_0(t_1,t_2,\ldots)
\]
of countably infinite transcendence degree.  A connected $k$-variety $V$ is
\emph{anisotropic modulo $2$} if every closed point of $V$ has even degree
over $k$.
\end{definition}

Both the terminology and the systematic use of flexible fields in isotropic
motivic categories originate in Vishik's work \cite{VishikIsotropic}; compare
\cite[Definitions~3.1--3.4]{TananiaCellular} for the stable setting.

Choose a set of representatives of the isomorphism classes of connected
anisotropic mod-$2$ varieties over $k$, and form their
coproduct $Q$ as a presheaf on $\Sm_k$.  Its augmented \v{C}ech simplicial
presheaf is
\[
  \check C(Q)_n=Q^{\times(n+1)},
\]
with face maps given by partial projections and degeneracy maps by partial
diagonals.  The augmentation is induced by $Q\to\Spec k$.
Throughout, $\Sigma^\infty_+\check C(Q)$ denotes the suspension spectrum of
the geometric realization of this simplicial presheaf.

\begin{definition}\label{def:isotropic-sphere}
The \emph{isotropic sphere} $\X$ is the cofiber in $\SH(k)$ of the augmented
\v{C}ech map:
\begin{equation}\label{eq:isotropic-cofiber}
  \Sigma^\infty_+\check C(Q)\longrightarrow\one\longrightarrow\X.
\end{equation}
The isotropic stable motivic homotopy category is the essential image of the
smashing localization $\X\wedge-$; it is denoted by $\SH(k/k)$.
\end{definition}

Tanania proves that $\X$ is an idempotent $E_\infty$-ring and that
\[
  \SH(k/k)\simeq\X\text{-}\Mod
\]
as stable $\infty$-categories; see
\cite[Proposition~6.1]{TananiaStable} and
\cite[Proposition~3.5]{TananiaCellular}.
Since $\X$ is an idempotent commutative algebra object, the standard
equivalence between $\X$-modules and $\X$-local objects is symmetric
monoidal; under it, $\wedge_{\X}$ corresponds to the ambient smash product
on $\X$-local objects.

The following well-known and elementary consequence of the \v{C}ech construction will be
used in the endomorphism calculation.

\begin{lemma}
\label{lem:anisotropic-killed}
If $V$ is a connected anisotropic mod-$2$ variety over $k$, then
$\X\wedge\Sigma^\infty_+V\simeq0.$
\end{lemma}

\begin{proof}
The inclusion
$i\colon V\hookrightarrow Q$ of this component gives a section
$v\mapsto(i(v),v)$ of the projection $Q\times V\to V$.  After base
change along $V\to\Spec k$, the augmented simplicial presheaf
$\check C(Q)\times V\to V$ is therefore the augmented \v{C}ech nerve of a
map with a section.  The standard extra-degeneracy identities
show that this augmented simplicial object is contractible over $V$.  Moreover,
\[
  \Sigma^\infty_+\check C(Q)\wedge\Sigma^\infty_+V
  \simeq
  \Sigma^\infty_+(\check C(Q)\times V).
\]
Smashing \eqref{eq:isotropic-cofiber} with $\Sigma^\infty_+V$ therefore gives
a cofiber sequence whose first map is an equivalence, and the last term is zero.
\end{proof}

\subsection{The isotropic Brown--Peterson object}

Let $\MBP$ be Vezzosi's motivic Brown--Peterson spectrum at the prime $2$
\cite{VezzosiMBP}, and put
\[
  E:=\X\wedge\MBP.
\]
Tanania proves that $E$ admits an $E_\infty$-ring structure
\cite[Proposition~7.1]{TananiaCellular}.  We use the following 
results from Sections~6--7 of that paper.

\begin{theorem}\label{thm:tanania-input}
Let $E=\X\wedge\MBP$.
\begin{enumerate}[label=\textup{(\roman*)}]
\item There is an isomorphism of bigraded rings
\[
  \pi_{**}(E)\cong\F,
\]
where the right-hand side is concentrated in bidegree $(0,0)$.
\item Every cellular $E$-module is a coproduct of bigraded suspensions of
$E$.
\item For cellular $E$-modules $Y$ and $Z$, the canonical map
\[
  [Y,Z]_E\longrightarrow
  \Hom_{\F}^{0,0}\bigl(\pi_{**}Y,\pi_{**}Z\bigr)
\]
is an isomorphism.  Equivalently, $\pi_{**}$ is an equivalence of
triangulated categories from cellular $E$-modules to bigraded
$\F$-vector spaces with the split triangulation described in
\cite[Remark~7.5]{TananiaCellular}; triangulated suspension shifts the first
bigrading index.
\end{enumerate}
\end{theorem}

\begin{proof}
Part (i) is \cite[Theorem~6.6]{TananiaCellular}.  The splitting in (ii) is
\cite[Proposition~7.2]{TananiaCellular}; full faithfulness and the resulting
triangulated equivalence are \cite[Corollary~7.3, Theorem~7.4 and
Remark~7.5]{TananiaCellular}.
\end{proof}

The compact part of this description requires a small argument that will be
used repeatedly.

\begin{lemma}\label{lem:compact-E-modules}
An object $Y$ of $E\text{-}\Mod_{\cell}$ is compact if and only if it is a
finite direct sum of bigraded suspensions of $E$.  Hence $\D$ is equivalent,
as a triangulated category, to the category of finite-dimensional bigraded
$\F$-vector spaces.
\end{lemma}

\begin{proof}
A finite direct sum of cells is compact.  Conversely, by
\cref{thm:tanania-input}, write
\[
  Y\simeq\bigoplus_{i\in I}\Sigma^{p_i,q_i}E.
\]
If $Y$ is compact, the chosen isomorphism from $Y$ to this coproduct factors
through a finite subcoproduct.  Composing the factorization with the inverse
isomorphism shows that $Y$ is a retract of a finite direct sum of cells.
After applying $\pi_{**}$, it becomes a retract of a finite-dimensional
bigraded vector space.  Full
faithfulness of $\pi_{**}$ lifts this decomposition to a finite direct sum of
bigraded suspensions of $E$.
\end{proof}

Since extension of scalars is exact, preserves retracts, and sends
$\Sigma^{p,q}\X$ to $\Sigma^{p,q}E$, it restricts to the functor
$F\colon\C\to\D$ used above.

\section{Connectivity of the isotropic sphere}\label{sec:connectivity}

We use Morel's homotopy $t$-structure on $\SH(k)$.  Its connective part is
the extension and colimit-closure of the spectra
$\Sigma^{p,q}\Sigma^\infty_+U$, where $U\in\Sm_k$ and $p-q\geq0$.  In
particular, if $Y$ is connective, then
\[
  \pi_{a,b}(Y)=0
  \qquad\text{for }a-b<0.
\]
Morel's stable $\mathbb A^1$-connectivity theorem
\cite{MorelConnectivity} says that the suspension spectrum of every smooth
scheme is connective; for arbitrary fields one may use Druzhinin's general
form specialized to $\Spec k$ \cite[Theorem~1]{DruzhininConnectivity}.
The connective aisle is closed under colimits and extensions; in particular,
it is closed under geometric realizations and under cofibers of maps between
connective spectra.

\begin{lemma}
\label{lem:presheaf-connective}
Let $P$ be any presheaf of sets on a fixed small skeleton of $\Sm_k$.  Then
$\Sigma^\infty_+P$ is connective.  The same holds levelwise and after
geometric realization for a simplicial presheaf of sets.
\end{lemma}

\begin{proof}
The canonical density presentation writes $P$ as the colimit of representable
presheaves over its category of elements:
\[
  P\cong\mathop{\operatorname{colim}}_{(X,x)\in\int P}h_X.
\]
The functor $\Sigma^\infty_+$ preserves colimits, each
$\Sigma^\infty_+X$ is connective by stable $\mathbb A^1$-connectivity, and
the connective aisle is closed under colimits.  This proves the first
assertion.  Applying the same argument in every simplicial degree and then
taking the geometric realization proves the second.
\end{proof}

\begin{proposition}
\label{prop:X-connective}
The isotropic sphere $\X$ is connective.  Consequently,
\[
  \pi_{a,b}(\X)=0
  \quad\text{whenever }a-b<0.
\]
\end{proposition}

\begin{proof}
Every level $\check C(Q)_n=Q^{\times(n+1)}$ is a presheaf of sets.  Hence Lemma 
\ref{lem:presheaf-connective} shows that
$\Sigma^\infty_+\check C(Q)_n$ is connective.  Geometric realization is a
colimit, so $\Sigma^\infty_+\check C(Q)$ is connective as well.  The sphere
$\one$ is connective.  Since the connective aisle is extension-closed, the
cofiber sequence \eqref{eq:isotropic-cofiber} implies that $\X$ is
connective.
\end{proof}

\section{Weight structures and the diagonal construction}\label{sec:weights}

We recall the conventions on weight structures that will be used throughout;
see \cite{BondarkoFoundations}.

\begin{definition}\label{def:weight-structure}
A \emph{weight structure} $w$ on an idempotent-complete triangulated category
$\mathcal T$ is a pair of retraction-closed full subcategories
$(\mathcal T_{w\le0},\mathcal T_{w\ge0})$ such that:
\begin{enumerate}[label=\textup{(\roman*)}]
\item $\mathcal T_{w\le0}\subseteq\mathcal T_{w\le0}[1]$ and
$\mathcal T_{w\ge0}[1]\subseteq\mathcal T_{w\ge0}$;
\item $\mathcal T_{w\le0}\perp\mathcal T_{w\ge1}$, where
$\mathcal T_{w\ge1}:=\mathcal T_{w\ge0}[1]$;
\item for every $M\in\mathcal T$ there is a distinguished triangle
\[
  A\longrightarrow M\longrightarrow B\longrightarrow A[1]
\]
with $A\in\mathcal T_{w\le0}$ and $B\in\mathcal T_{w\ge1}$.
\end{enumerate}
We put
\[
  \mathcal T_{w\le n}:=\mathcal T_{w\le0}[n],
  \quad
  \mathcal T_{w\ge n}:=\mathcal T_{w\ge0}[n],
  \quad
  \mathcal T_{w=n}:=\mathcal T_{w\le n}\cap\mathcal T_{w\ge n}.
\]
The \emph{heart} is
$\mathcal H_w:=\mathcal T_{w=0}$.  The weight structure is \emph{bounded} if
for every $M$ there are integers $a\le b$ with
$M\in\mathcal T_{w\ge a}\cap\mathcal T_{w\le b}$; this intersection is
denoted by $\mathcal T_{w\in[a,b]}$.
\end{definition}

A functor between weighted triangulated categories is \emph{weight-exact} if
it preserves both halves of the weight structures.  A full additive
subcategory $\B\subseteq\mathcal T$ is \emph{negative} if
\[
  \Hom_{\mathcal T}(B,B'[n])=0
  \qquad(B,B'\in\B,\ n>0).
\]
It \emph{densely generates} $\mathcal T$ if its thick 
triangulated closure is all of $\mathcal T$.

We use the following standard construction theorem.

\begin{proposition}
\label{prop:negative-generates}
If a negative additive subcategory $\B$ densely generates an
idempotent-complete triangulated category $\mathcal T$, then $\mathcal T$
admits a unique bounded weight structure for which $\B$ is contained in the
heart.  Its heart is $\Kar_{\mathcal T}(\B)$.
\end{proposition}

\begin{proof}
This is \cite[Corollary~2.1.2]{BondarkoSosnilo}.
\end{proof}

Put
\[
  \B:=\Add\{T^q\mid q\in\mathbb Z\}\subseteq\C.
\]

\begin{proposition}
\label{prop:diagonal-weight}
The category $\C$ admits a bounded weight structure $w$ whose heart is
\[
  \Hh:=\C_{w=0}=\Kar_{\C}(\B).
\]
It is characterized by declaring every $T^q$ to have weight zero.  The tensor
product is weight-exact in the sense that
\[
  \C_{w\le0}\wedge_{\X}\C_{w\le0}\subseteq\C_{w\le0},
  \qquad
  \C_{w\ge0}\wedge_{\X}\C_{w\ge0}\subseteq\C_{w\ge0}.
\]

There is likewise a bounded weight structure on $\D$ whose heart is
\[
  \V:=\Add\{U^q\mid q\in\mathbb Z\}.
\]
Under $\pi_{**}$, $\V$ identifies with finite-dimensional
$\mathbb Z$-graded $\F$-vector spaces, where $U^q$ corresponds to the
one-dimensional object $\F(q)$, namely $\F$ placed in Tate degree $q$.
The functor $F\colon\C\to\D$ is
weight-exact, and its restriction to the hearts is
\[
  \Phi\colon\Hh\longrightarrow\V,
  \qquad H\longmapsto E\wedge_{\X}H.
\]
\end{proposition}

\begin{proof}
For $a,b\in\mathbb Z$ and $n>0$, the free--forgetful adjunction for
$\X$-modules gives
\begin{align*}
  \Hom_{\C}(T^a,T^b[n])
  &\cong
  [\Sigma^{a,a}\one,\Sigma^{b+n,b}\X]_{\SH(k)}\\
  &\cong\pi_{a-b-n,a-b}(\X).
\end{align*}
The difference of the two indices is $-n<0$, so this group vanishes by Proposition 
\ref{prop:X-connective}, so $\B$ is negative.  By
\eqref{eq:suspension-convention}, every cellular generator is a shift of an
object of $\B$.  Thus $\B$ densely generates $\C$, and Proposition 
\ref{prop:negative-generates} gives the asserted bounded weight structure
and its heart.

Since $T^a\wedge_{\X}T^b\simeq T^{a+b}$, tensor products of generators of the
heart remain in the heart.  More explicitly,
\cite[Corollary~2.1.2(3)]{BondarkoSosnilo} identifies
$\C_{w\le0}$ with the extension--retraction closure of the objects
$T^q[i]$ for $i\le0$, and $\C_{w\ge0}$ with the corresponding closure for
$i\ge0$.  Since
\[
  T^a[i]\wedge_{\X}T^b[j]\simeq T^{a+b}[i+j],
\]
the relevant sets of shifts are closed under tensor product.  Exactness of
tensor product in each variable and preservation of retracts now prove the
two inclusions.

For $\D$, one has
\[
  \Hom_{\D}(U^a,U^b[n])
  \cong\pi_{a-b-n,a-b}(E)=0
  \qquad(n>0)
\]
by \cref{thm:tanania-input}.  The objects $U^q$ therefore form a negative
generating family.  The construction theorem first identifies the heart with
the Karoubi closure of $\Add\{U^q\mid q\in\mathbb Z\}$.  Under $\pi_{**}$,
every homogeneous idempotent splits degreewise in a finite-dimensional
bigraded vector space; hence this Karoubi closure is already $\V$.  Finally,
$F(T^q[i])=U^q[i]$.  The generated descriptions of the two weight structures
therefore show directly that $F$ preserves both halves and hence is
weight-exact.
\end{proof}

We shall use the following form of Bondarko's weight-detection theorem.

\begin{proposition}\label{prop:weight-detection}
Let $G\colon\mathcal T\to\mathcal T'$ be a weight-exact functor between
triangulated categories with bounded weight structures.  If the induced
functor on hearts is full and conservative, then:
\begin{enumerate}[label=\textup{(\roman*)}]
\item $G$ is conservative;
\item $G$ detects upper and lower weight bounds;
\item if $G(M)$ is pure of weight $n$, then $M$ is pure of weight $n$.
\end{enumerate}
\end{proposition}

\begin{proof}
  The first two assertions are
\cite[Theorem~1.5.1(1),(2)]{BondarkoWeightComplexes}.  For (iii), apply (ii)
to the two bounds $\le n$ and $\ge n$.
\end{proof}

For later use we also record the strong weight-complex construction.  Equip
$K^b(\Hh)$ with the stupid weight structure, whose heart is identified with
$\Hh$ by complexes concentrated in degree zero.  Since $\Cinf$ is a stable
$\infty$-category and $w$ is bounded, Sosnilo constructs, uniquely up to
equivalence, an exact
weight-exact functor
\begin{equation}\label{eq:strong-weight-complex}
  t\colon\C\longrightarrow K^b(\Hh)
\end{equation}
which restricts to the identity on $\Hh$; see
\cite[Section~3 and Corollary~3.5]{Sosnilo}.  Similarly there is
$t_E\colon\D\to K^b(\V)$.  The construction is functorial for 
weight-exact functors between the enhanced categories.  We choose $t$ and
$t_E$ coherently with this functoriality.  Applied to $F$, it gives a natural
isomorphism
\begin{equation}\label{eq:weight-complex-base-change}
  K^b(\Phi)\,t(M)\xrightarrow{\;\sim\;}t_E(FM)
\end{equation}
naturally in $M$.  The projection of $t$ to Bondarko's weak homotopy category
is the weak weight complex defined from weight Postnikov towers.  The
distinction between strong and weak complexes will matter in
\cref{sec:nilpotence}.

\section{The endomorphism ring of the isotropic unit}\label{sec:end-unit}

Let us recall that Morel identifies $\pi_{0,0}(\one)$ with the Grothendieck--Witt ring
$\GW(k)$ of nondegenerate symmetric bilinear forms over $k$, with orthogonal
sum as addition and tensor product as multiplication
\cite{MorelPiZero}.  This identification is valid for arbitrary fields; see
also \cite[footnote~1]{HoyoisTrace}.
The following calculation is the key input into the structure of the
weight heart.
We first isolate the quadratic trace relation.

\begin{lemma}\label{lem:quadratic-trace}
Let $R:=\pi_{0,0}(\X).$
For every nonsquare $a\in k^\times$, the image in $R$ of the Grothendieck--Witt class $1+\langle a\rangle$ is zero.
\end{lemma}

\begin{proof}
Let $L=k(\sqrt a)$.  Since $a$ is a nonsquare and
$\operatorname{char}(k)\ne2$, the scheme $\Spec L$ is connected finite
\'etale of degree $2$.  Its unique closed point has even degree, so it is
anisotropic mod $2$.  By Lemma \ref{lem:anisotropic-killed},
$\X\wedge\Sigma^\infty_+\Spec L\simeq0.$ The suspension spectrum of a finite \'etale scheme is strongly dualizable.
The symmetric monoidal functor $\X\wedge-$ preserves categorical traces;
therefore the image in $R=\End_{\X\text{-}\Mod}(\X)$ of the motivic Euler
characteristic $\chi(\Spec L)$ is zero.

By Hoyois's quadratic trace formula
\cite[Theorem~1.9]{HoyoisTrace}, in the case of zero-dimensional $L-$vector space $V=0$,
$\chi(\Spec L)$ is the Scharlau transfer of
$\langle1\rangle$, equivalently the trace form
\[
  (x,y)\longmapsto\operatorname{Tr}_{L/k}(xy).
\]
With respect to the basis $1,\sqrt a$, the matrix of this form is
$\operatorname{diag}(2,2a)$.  Hence
\[
  \chi(\Spec L)=\langle2\rangle+\langle2a\rangle
  =\langle2\rangle(1+\langle a\rangle)
  \quad\text{in }\GW(k).
\]
The class $\langle2\rangle$ is a multiplicative unit, with inverse
$\langle1/2\rangle$.  Its image in $R$ is therefore a unit, so the vanishing
of $\chi(\Spec L)$ in $R$ implies
$1+\langle a\rangle=0$.
\end{proof}

\begin{lemma}\label{lem:three-nonsquares}
Each of $t_1,t_2,t_1t_2\in k^\times$ is a nonsquare.
\end{lemma}

\begin{proof}
If one of the three displayed elements were a square, a chosen square root
would belong to a rational function field in finitely many of the variables.
On that finite-variable subfield, the discrete valuation associated with the
prime divisor $t_i=0$ assigns an even value to every square.  On the other
hand,
\[
  v_{t_1}(t_1)=1,\qquad
  v_{t_2}(t_2)=1,\qquad
  v_{t_1}(t_1t_2)=v_{t_2}(t_1t_2)=1.
\]
Hence $t_1$, $t_2$, and $t_1t_2$ cannot be squares in $k$.
\end{proof}

\begin{proposition}\label{prop:end-unit}
The unit homomorphism $\mathbb Z\to\pi_{0,0}(\X)$ factors through an
isomorphism
\[
  \F\xrightarrow{\,\sim\,}\pi_{0,0}(\X).
\]
Consequently,
\[
  \End_{\C}(\X)\cong\pi_{0,0}(\X)\cong\F,
\]
and, for every $q\in\mathbb Z$, the map
\[
  \End_{\C}(T^q)\longrightarrow\End_{\V}(U^q)
\]
induced by $\Phi$ is an isomorphism.
\end{proposition}

\begin{proof}
The free--forgetful adjunction for $\X$-modules identifies
$\End_{\C}(\X)$ with $R$.  By Lemma \ref{lem:presheaf-connective}, $\pi_{-1,0}(\Sigma^\infty_+\check C(Q))=0$.  The long exact sequence of
\eqref{eq:isotropic-cofiber} therefore gives a surjection
\begin{equation}\label{eq:GW-surjection}
  \pi_{0,0}(\one)\twoheadrightarrow R.
\end{equation}
  Thus
\eqref{eq:GW-surjection} becomes a surjective ring homomorphism
$\GW(k)\twoheadrightarrow R$.

By Lemmas \ref{lem:quadratic-trace} and \ref{lem:three-nonsquares},
\[
  \langle t_1\rangle=\langle t_2\rangle
  =\langle t_1t_2\rangle=-1
  \quad\text{in }R.
\]
  On the other hand, multiplication of one-dimensional
forms gives
\[
  \langle t_1t_2\rangle
  =\langle t_1\rangle\langle t_2\rangle=(-1)^2=1.
\]
Hence $2=0$ in $R$.

We claim that every one-dimensional form has image $1$ in $R$.  If
$u\in k^\times$ is a square, then $\langle u\rangle=1$ already in
$\GW(k)$.  If $u$ is a nonsquare, Lemma \ref{lem:quadratic-trace} gives
$\langle u\rangle=-1=1$ in $R$.  Since every nondegenerate symmetric
bilinear form over a field of characteristic different from $2$ is
orthogonally diagonalizable, the map $\GW(k)\to R$ depends only on the rank
modulo $2$.  Since $\GW(k)$ is the group completion of the monoid of such
forms, this determines the map on all of $\GW(k)$.  It factors as a
surjection
\[
  \GW(k)\xrightarrow{\ \rk\bmod2\ }\F\twoheadrightarrow R.
\]

The structural map $\X\to E$ induces a unital ring map
\[
  R\longrightarrow\pi_{0,0}(E)=\F
\]
by \cref{thm:tanania-input}.  The composite
$\F\twoheadrightarrow R\to\F$ is the unique unital endomorphism of $\F$,
so it is the identity.  Hence $\F\to R$ is both surjective and injective.

Finally, tensoring with the invertible object $T^q$ identifies
$\End_{\C}(T^q)$ with $\End_{\C}(\X)=R$.  Similarly,
$\End_{\V}(U^q)=\F$.  Under these identifications, the map induced by
$\Phi$ is the unital isomorphism $R\to\F$ just obtained.
\end{proof}

\begin{corollary}
\label{cor:F2-linear}
The additive category $\C$ is canonically $\F$-linear.  In particular,
every morphism group in $\C$ has exponent $2$.
\end{corollary}

\begin{proof}
In any symmetric monoidal additive category, $\End(\one)$ acts centrally on
every morphism group.  By Proposition \ref{prop:end-unit}, the scalar $2\id_{\X}$ is
zero.  Its action on an $\X$-module $M$ is $2\id_M$, so $2\id_M=0$ for every
$M\in\C$ and hence $2f=0$ for every morphism $f$.
\end{proof}

We got a rather amusing result:
the three explicit nonsquares sharpen the quadratic-transfer conclusion in
this argument; compare \cite[Remark~3.6]{TananiaCellular}.  A single quadratic
transfer only produces the relation $1+\langle a\rangle=0$.  The multiplicative identity
$\langle t_1t_2\rangle=\langle t_1\rangle\langle t_2\rangle$ converts the
three relations into $2=0$, which then collapses the whole image of
$\GW(k)$ to rank modulo $2$.

\section{The weight heart and reduction to \texorpdfstring{$E$}{E}}
\label{sec:heart}

Although $\Hh$ is defined as a Karoubi closure, it contains no exotic
retracts.  The proof begins with a finite matrix calculation.

Let
\[
  Q_0=\bigoplus_{i=1}^{N}T^{q_i},
  \qquad N\ge1,
\]
and let
\[
  s(Q_0):=\#\{q_i\mid 1\le i\le N\}
\]
be the number of distinct Tate degrees, 
occurring in $Q_0$ (multiplicities are ignored).  Put
\[
  J_{Q_0}:=\ker\!\left(
  \End_{\C}(Q_0)\longrightarrow\End_{\V}(\Phi Q_0)
  \right).
\]

\begin{lemma}\label{lem:matrix-kernel}
One has
\[
  J_{Q_0}^{s(Q_0)}=0.
\]
\end{lemma}

\begin{proof}
Write an endomorphism of $Q_0$ as an $N\times N$ matrix whose $(j,i)$ entry
is a morphism $T^{q_i}\to T^{q_j}$.  If an endomorphism belongs to
$J_{Q_0}$, then every matrix entry has zero image under $\Phi$, because its
whole reduced matrix is zero and reduction preserves the biproduct
projections and inclusions.

Set $s=s(Q_0)$ and consider a product of $s$ elements of $J_{Q_0}$.  Each matrix entry of the
product is a sum of composites indexed by paths
\[
  i_0\longrightarrow i_1\longrightarrow\cdots\longrightarrow i_s
\]
through the summands of $Q_0$.  Among the $s+1$ labels
$q_{i_0},\ldots,q_{i_s}$, two coincide; say $q_{i_r}=q_{i_t}$ with
$r<t$.  If $x_j\colon T^{q_{i_{j-1}}}\to T^{q_{i_j}}$ denotes the matrix
entry chosen at the $j$th step, the corresponding nonempty subcomposite is
\[
  x_t\cdots x_{r+1}\colon
  T^{q_{i_r}}\longrightarrow T^{q_{i_t}}=T^{q_{i_r}}.
\]
It is therefore an element of $\End_{\C}(T^{q_{i_r}})$.  Its image under
$\Phi$ is zero, since all the
factors along the subpath have zero reduced matrix entries.  By Proposition
\ref{prop:end-unit}, $\Phi$ is injective on this endomorphism ring, so the
subcomposite is zero.  Every path term therefore vanishes, proving
$J_{Q_0}^{s}=0$.
\end{proof}

We shall use the following elementary ring-theoretic form of idempotent
lifting.

\begin{lemma}
\label{lem:idempotent-conjugacy}
Let $A$ be a ring, $J\subseteq A$ a nilpotent two-sided ideal, and let
$e,e_0\in A$ be idempotents with $e-e_0\in J$.  Then $e$ and $e_0$ are
conjugate by a unit of $1+J$.
\end{lemma}

\begin{proof}
Set
\[
  v:=e_0e+(1-e_0)(1-e).
\]
A direct calculation gives $e_0v=ve$.  Modulo $J$ one has $e=e_0$, and hence
\[
  v\equiv e_0^2+(1-e_0)^2=1\pmod J.
\]
Thus $v=1+j$ for some nilpotent element $j\in J$. Therefore
$e_0=vev^{-1}$.
\end{proof}

\begin{proposition}
\label{prop:no-exotic-retracts}
Every object of $\Hh$ is isomorphic to a finite direct sum of diagonal
spheres.  Equivalently,
\[
  \Hh=\Add\{T^q\mid q\in\mathbb Z\}
\]
up to equivalence of additive categories.  Moreover, $\Hh$ is
Krull--Schmidt: the objects $T^q$ are pairwise nonisomorphic and are precisely
its indecomposable objects, and the multiplicity of every $T^q$ is unique.
\end{proposition}

\begin{proof}
An object of $\Hh=\Kar_{\C}(\B)$ is the image of an idempotent
$e\in\End_{\C}(Q_0)$ for some finite sum
$Q_0=\bigoplus_iT^{q_i}$.  The idempotent
$\bar e=\Phi(e)$ acts degreewise on the finite-dimensional graded vector
space $\Phi(Q_0)$.  In each internal degree, choose bases of its image and
kernel.  The resulting degree-preserving change of basis conjugates
$\bar e$ to a standard diagonal projection.

Every equal-degree matrix block lifts uniquely to a matrix of maps between
copies of $T^q$, because Proposition \ref{prop:end-unit} identifies
$\End(T^q)$ with $\F$.  Lift the inverse change of basis matrix as well.
The two products are the identity in every equal-degree block by the
injectivity in Proposition \ref{prop:end-unit}; hence the two lifts are inverse
automorphisms of $Q_0$.  After conjugating $e$ by this lift, we may assume
that $e-e_0\in J_{Q_0},$
where $e_0$ is a standard diagonal idempotent.  The ideal $J_{Q_0}$ is
nilpotent by Lemma \ref{lem:matrix-kernel}; hence Lemma 
\ref{lem:idempotent-conjugacy} shows that $e$ and $e_0$ are conjugate.
Their images are isomorphic, and the image of $e_0$ is a finite direct sum of
some of the original diagonal summands.

Finally, $\Phi(T^q)=U^q$ shows that the $T^q$ are pairwise nonisomorphic,
while $\End(T^q)=\F$ makes each of them indecomposable.  The graded dimension
of $\Phi(H)$ records the multiplicity of every $T^q$ in a decomposition of
$H$, proving uniqueness.
\end{proof}

\begin{corollary}\label{cor:heart-functor}
The functor $\Phi\colon\Hh\to\V$ is full and conservative, reflects the
multiplicity of each $T^q$, and has nilpotent kernel on the endomorphism ring
of every object of $\Hh$.
\end{corollary}

\begin{proof}
By Proposition \ref{prop:no-exotic-retracts}, write source and target objects as finite
direct sums of the $T^q$.  In $\V$, maps between different internal degrees
vanish, while maps between equal-degree summands are matrices over $\F$.
Every such matrix lifts entrywise by Proposition \ref{prop:end-unit}, proving fullness.
The reduced graded vector space records exactly the number of summands of
each $T^q$, proving multiplicity reflection and object conservativity.

For conservativity on morphisms, let $f\colon H\to H'$ have invertible image.
By fullness, lift the inverse of $\Phi(f)$ to a morphism
$g\colon H'\to H$.  Then $gf-1_H$ and $fg-1_{H'}$ belong to nilpotent
endomorphism kernels, by Lemma \ref{lem:matrix-kernel}.  Hence $gf$ and $fg$ are
invertible .  Thus
\[
  (gf)^{-1}g
  \quad\text{and}\quad
  g(fg)^{-1}
\]
are respectively a left and a right inverse of $f$; they coincide, so $f$ is
invertible.  The same matrix kernel lemma gives the final nilpotence
assertion.
\end{proof}

\begin{remark}
The category $\V$ is a local semisimple symmetric monoidal abelian category:
under its graded tensor product, the tensor product of two nonzero
finite-dimensional graded $\F$-vector spaces is nonzero.  Thus Propositions 
\ref{prop:diagonal-weight}, \ref{prop:no-exotic-retracts} and Corollary \ref{cor:heart-functor} verify
the hypotheses of \cite[Theorem~1.1]{BondarkoTabuadaPicard}.
\end{remark}

\section{Minimal weight complexes}\label{sec:minimal-complexes}

The preceding description of the heart permits a familiar minimal complex
argument, analogous to cancellation over a local or semisimple residue
category.

\begin{definition}\label{def:Phi-minimal}
A bounded complex $C^\bullet=(C^i,d^i)$ in $\Hh$ is
\emph{$\Phi$-minimal} if
\[
  \Phi(d^i)=0
  \quad\text{for every }i.
\]
Equivalently, $K^b(\Phi)(C^\bullet)$ is a complex with zero differential.
\end{definition}

\begin{lemma}\label{lem:cancellation}
Every bounded complex in $\Hh$ is homotopy equivalent to a
$\Phi$-minimal bounded complex.  The procedure terminates after finitely many
steps and never enlarges the cohomological support.
\end{lemma}

\begin{proof}
Use Proposition \ref{prop:no-exotic-retracts} to decompose every term into a finite
direct sum of objects $T^q$.  Suppose that $\Phi(d^i)\ne0$.  For some $q$,
the degree-$q$ block of $\Phi(d^i)$ has positive rank.  Degreewise changes of
bases in $C^i$ and $C^{i+1}$ lift to automorphisms in $\Hh$ by Proposition 
\ref{prop:end-unit}.  After such changes, the differential contains a
component
\[
  u\colon T^q\longrightarrow T^q
\]
whose image under $\Phi$ is $1$.  The endomorphism map of Proposition
\ref{prop:end-unit} is an isomorphism, so $u=1_{T^q}$ and in particular is
invertible.

Write the relevant differential in block form
\[
  d^i=
  \begin{pmatrix}
    u & a\\ b & c
  \end{pmatrix}.
\]
Multiplication by elementary triangular automorphisms of source and target
replaces it by
\[
  \begin{pmatrix}
    u & 0\\ 0 & c-bu^{-1}a
  \end{pmatrix}.
\]
The equations $d^id^{i-1}=0$ and $d^{i+1}d^i=0$ then force all components of
the adjacent differentials entering the source of $u$ or leaving its target
to vanish.  Thus the whole complex splits as a direct sum of the contractible
two-term complex
\[
  T^q\xrightarrow{\,u\,}T^q
\]
and a bounded complex having two fewer diagonal summands.  Removing this
contractible summand strictly decreases the total number of summands in all
terms.  Repetition therefore terminates, and at termination every reduced
differential is zero.
\end{proof}

\begin{proposition}
\label{prop:minimal-uniqueness}
Two $\Phi$-minimal bounded complexes in $\Hh$ that are homotopy equivalent
are isomorphic as complexes.  Consequently a $\Phi$-minimal representative
of $t(M)$ is unique up to (noncanonical) chain isomorphism.
\end{proposition}

\begin{proof}
Let $f\colon C^\bullet\to D^\bullet$ be a chain homotopy equivalence between
$\Phi$-minimal complexes.  After applying $\Phi$, both differentials are
zero.  Hence the equations asserting that $\Phi(f)$ has a homotopy inverse
reduce to the statement that every $\Phi(f^i)$ is an isomorphism.  The heart
functor is conservative by Corollary  \ref{cor:heart-functor}, so every $f^i$ is an
isomorphism.  A termwise invertible chain map is a chain isomorphism. 
\end{proof}

We also need a uniform nilpotence statement for maps of a minimal complex.

\begin{lemma}
\label{lem:minimal-complex-kernel}
Let $C^\bullet$ be a nonzero $\Phi$-minimal bounded complex and put
\[
  d(C^\bullet):=\max_i\#\{q\mid T^q\text{ occurs in }C^i\}.
\]
Then the
ideal
\[
  R_{C^\bullet}:=
  \ker\!\left(
  \End_{K^b(\Hh)}(C^\bullet)
  \longrightarrow
  \End_{K^b(\V)}(K^b(\Phi)C^\bullet)
  \right)
\]
satisfies
\[
  R_{C^\bullet}^{\,d(C^\bullet)}=0.
\]
\end{lemma}

\begin{proof}
Put $d=d(C^\bullet)$, and choose arbitrary
$x_1,\ldots,x_d\in R_{C^\bullet}$, represented by chain maps
$c_1,\ldots,c_d$.  The reduced differential is zero, and each reduced chain
map $\Phi(c_j)$ is null-homotopic.  The chain-homotopy equation therefore
gives $\Phi(c_j^i)=0$ for all $i,j.$
For each term $C^i$, the kernel of
\[
  \End_{\Hh}(C^i)\longrightarrow\End_{\V}(\Phi C^i)
\]
has $d^{th}$ power zero by Lemma \ref{lem:matrix-kernel}; repetitions of a fixed
$T^q$ do not affect the bound.  Thus
$(c_d\cdots c_1)^i=0$
for every $i$.  Hence $x_d\cdots x_1=0$.
\end{proof}

\section{The finite cellular \texorpdfstring{$\MBP$}{MBP}--Nakayama theorem}
\label{sec:nakayama}

We now prove the detection statements from the introduction by weight
detection.

\begin{proof}[Proof of \cref{thm:pure-nakayama}]
The functor on hearts is full and conservative by Corollary 
\ref{cor:heart-functor}.  Therefore Proposition  \ref{prop:weight-detection} applies to
the weight-exact functor $F\colon\C\to\D$.

If $E_{**}(M)=0$, then $F(M)\simeq0$ by the full faithfulness and object
classification in \cref{thm:tanania-input}.  Conservativity of $F$ gives
$M\simeq0$.

Suppose now that $E_{**}(M)$ is concentrated in diagonal degree $n$.  Since
$M$ is finite cellular and $F$ is exact and preserves retracts, $F(M)$ is a
compact cellular $E$-module.  By Lemma  \ref{lem:compact-E-modules}, it is a finite
direct sum of suspensions of $E$.  Concentration in diagonal degree $n$
forces the decomposition
\[
  F(M)\simeq
  \bigoplus_{q\in\mathbb Z}
  \bigl(U^q[n]\bigr)^{\oplus m_q},
\]
where $m_q=\dim_{\F}E_{n+q,q}(M)$ and only finitely many $m_q$ are nonzero.
The right-hand side is pure of weight $n$.  Weight detection implies that
$M$ is pure of weight $n$, so
 $M\simeq H[n]$
for some $H\in\Hh$.  Applying $F$ and cancelling the common shift yields
\[
  \Phi(H)\simeq
  \bigoplus_q(U^q)^{\oplus m_q}.
\]
By multiplicity reflection in Corollary  \ref{cor:heart-functor},
\[
  H\simeq\bigoplus_q(T^q)^{\oplus m_q}.
\]
Using \eqref{eq:suspension-convention}, we obtain
\[
  M\simeq
  \bigoplus_q\bigl(T^q[n]\bigr)^{\oplus m_q}
  =
  \bigoplus_q
  \bigl(\Sigma^{n+q,q}\X\bigr)^{\oplus m_q}.
\]
\end{proof}

\begin{proof}[Proof of Corollary \ref{cor:one-cell-nakayama}]
Apply \cref{thm:pure-nakayama}(ii) with $n=p-q$ and with the unique nonzero
multiplicity equal to $m_q=1$.
\end{proof}

\begin{proof}[Proof of Corollary \ref{cor:equivalence-detection}]
The forward implication is immediate.  Conversely, if $E_{**}(f)$ is an
isomorphism, then the cofiber of $f$ has zero $E$-homology.  It vanishes by
\cref{thm:pure-nakayama}(i), so $f$ is an isomorphism.
\end{proof}

\section{Computation of the Picard group and $K_0$}\label{sec:picard}

We first compute the Picard group after extension of scalars.

\begin{lemma}
\label{lem:E-picard}
Every tensor-invertible object of $\D$ is isomorphic to a unique bigraded
suspension of $E$.  Consequently,
\[
  \Pic(\D)\cong\mathbb Z^2,
  \qquad
  (p,q)\longmapsto[\Sigma^{p,q}E].
\]
\end{lemma}

\begin{proof}
The assertion follows immediately from Lemma \ref{lem:compact-E-modules} and multiplicativity of the total $\F$-dimension of $\pi_{**}$ under tensor product.
\end{proof}

\begin{proof}[Proof of \cref{thm:main-picard}]
Let $P\in\C$ be tensor-invertible.  Since $F$ is symmetric monoidal, $F(P)$
is tensor-invertible in $\D$.  By Lemma \ref{lem:E-picard}, there is a unique
pair $(p,q)$ such that
\[
  F(P)\simeq\Sigma^{p,q}E.
\]
Equivalently,
\[
  E_{**}(P)\cong\Sigma^{p,q}\F.
\]
The one-cell Nakayama \ref{cor:one-cell-nakayama} gives
$P\simeq\Sigma^{p,q}\X.$
Thus the homomorphism $\mathbb Z^2\to\Pic(\C)$ is surjective.

For injectivity, suppose
$\Sigma^{p,q}\X\simeq\Sigma^{p',q'}\X$.  Extension of scalars gives
\[
  \Sigma^{p,q}E\simeq\Sigma^{p',q'}E.
\]
Taking $\pi_{**}$ shows $(p,q)=(p',q')$.  Therefore the displayed
homomorphism is an isomorphism.
\end{proof}

Thus extension of scalars identifies $\Pic(\C)$ with $\Pic(\D)$; under the
two identifications with $\mathbb Z^2$, it is the identity.

For an essentially small triangulated category $\mathcal T$, write
$K_0(\mathcal T)$ for the abelian group generated by isomorphism classes
$[M]$, modulo the relations $[B]=[A]+[C]$ for distinguished triangles
$A\to B\to C\to A[1]$.  An exact tensor product makes this a ring by
$[M][N]=[M\otimes N]$.

\begin{theorem}
\label{thm:K0}
There are canonical ring isomorphisms
\[
  K_0(\C)\xrightarrow{\;\sim\;}K_0(\D)
  \xrightarrow{\;\sim\;}\mathbb Z[u,u^{-1}],
\]
where $u=[T]$ on the left and $u=[U]$ on the right.  Explicitly,
\[
  [\Sigma^{p,q}\X]=(-1)^{p-q}u^q
\]
and every $M\in\C$ satisfies
\begin{equation}\label{eq:K0-Euler}
  [M]=\sum_{p,q\in\mathbb Z}
  (-1)^{p-q}\dim_{\F}E_{p,q}(M)\,u^q.
\end{equation}

\end{theorem}

\begin{proof}
For a bounded weight structure, the weight-complex functor induces an
isomorphism between the Grothendieck group and the split Grothendieck group
of the heart by \cite[Proposition~4.2]{Sosnilo}; equivalently, it sends an
object to the alternating sum of the terms of a weight complex.  Since the
heart is closed under tensor products, the inverse map induced by its
inclusion is multiplicative, so this is an isomorphism of rings.  By Proposition 
\ref{prop:no-exotic-retracts},
\[
  K_0(\C)\cong K_0^{\mathrm{split}}(\Hh)
  \cong\bigoplus_{q\in\mathbb Z}\mathbb Z[T^q]
  \cong\mathbb Z[u,u^{-1}].
\]
The same argument in $\D$ uses its split diagonal heart and gives the second
copy of $\mathbb Z[u,u^{-1}]$.  The heart functor sends $T^q$ to $U^q$, so
$F$ induces the identity under these identifications.  Since
$\Sigma^{p,q}\X=T^q[p-q]$, the suspension relation in $K_0$ gives the stated
formula for spheres.

Finally, Theorem \ref{thm:tanania-input} and Lemma \ref{lem:compact-E-modules} gives a finite
decomposition
\[
  F(M)\simeq
  \bigoplus_{p,q\in\mathbb Z}
  \bigl(\Sigma^{p,q}E\bigr)^{\oplus\dim_{\F}E_{p,q}(M)}.
\]
Its class in $K_0(\D)$ is the right-hand side of
\eqref{eq:K0-Euler}.  Since $K_0(F)$ is the isomorphism just computed, the
same formula holds in $K_0(\C)$.
\end{proof}

\section{Weight profiles and composition nilpotence detected by isotropic
\texorpdfstring{$\MBP$}{MBP}-homology}
\label{sec:nilpotence}

Object detection alone does not imply nilpotence of maps.  We now use the
stable $\infty$-categorical enhancement $\Cinf$ and distinguish the strong
weight complex from Bondarko's weak weight complex.

\begin{proof}[Proof of \cref{thm:weight-profile}]
By Theorem \ref{thm:tanania-input} and Lemma \ref{lem:compact-E-modules}, there is a finite
decomposition
\[
  F(M)\cong
  \bigoplus_{n,q}(U^q[n])^{\oplus\beta_{n,q}(M)}.
\]
For the diagonal weight structure on $\D$, this object has weights in
$[a,b]$ exactly when all multiplicities with $n\notin[a,b]$ vanish.  The
heart functor $\Phi$ is full and conservative by Corollary
\ref{cor:heart-functor}; hence Proposition  \ref{prop:weight-detection}, applied to
$F$, says that $F$ detects both weight bounds.  This proves the asserted
equivalence and the minimality of $[a(M),b(M)]$.

Let $C^\bullet$ be a $\Phi$-minimal representative of $t(M)$.  Naturality
of the strong weight complex gives
\[
  K^b(\Phi)(C^\bullet)
  \cong t_E(FM)
  \cong \bigoplus_{n,q}\F(q)^{\oplus\beta_{n,q}(M)}[n].
\]
Both sides have zero differential.  An isomorphism between zero-differential
complexes in $K^b(\V)$ is termwise an isomorphism.  Multiplicity reflection
in Corollary  \ref{cor:heart-functor} therefore yields
\[
  C^{-n}\cong
  \bigoplus_q(T^q)^{\oplus\beta_{n,q}(M)}
\]
for every $n$.  Every other $\Phi$-minimal representative is chain-isomorphic
to $C^\bullet$ by Proposition  \ref{prop:minimal-uniqueness}, so the formula holds for
every such representative.
\end{proof}

\subsection{Weight Postnikov towers and weak homotopy}

Let $M\in\C_{w\in[a,b]}$.  A \emph{weight Postnikov tower} for $M$ in the
range $[a,b]$ consists of objects and maps
\[
  0=M_{\le a-1}\longrightarrow M_{\le a}
  \longrightarrow\cdots\longrightarrow M_{\le b}=M
\]
together with distinguished triangles
\[
  M_{\le i-1}\longrightarrow M_{\le i}
  \longrightarrow H_i[i]\longrightarrow M_{\le i-1}[1],
  \qquad H_i\in\Hh.
\]
Such a tower is obtained by iterating weight decompositions.  The objects
$H_i[i]$ are its \emph{weight factors}.  The connecting maps between
successive triangles define a bounded complex of objects of $\Hh$, well
defined up to Bondarko's weak homotopy relation.  We use cohomological
indexing and place $H_i$ in degree $-i$.  Consequently, if the tower has
weights in $[a,b]$, its weight complex is supported in degrees $[-b,-a]$.
This sign convention is compatible with $t(H[n])=H[n]$ for $H\in\Hh$.

Recall that, two chain maps $c,c'\colon A^\bullet\to B^\bullet$ are
\emph{weakly homotopic} if there exist degree $-1$ graded maps
$h,j\colon A^\bullet\to B^{\bullet-1}$ such that
\[
  c-c'=d_Bh+jd_A.
\]
Ordinary chain homotopy is the special case $h=j$ (with the usual sign
convention incorporated in the differentials).  Quotienting morphisms by
weak homotopy gives weak homotopy category.

The functoriality and strictification properties of these towers that we need
are collected, with precise references, in Lemma  \ref{lem:strictification} below.
With the cohomological convention just fixed---$H_i$ lies in degree $-i$---
Sosnilo's comparison identifies the weak-homotopy class of the complex of
weight factors with the projection of the strong weight complex
\eqref{eq:strong-weight-complex}.

If $M\in\C_{w\in[a,b]}$, the strong weight complex $t(M)$ can therefore be
represented by a bounded complex supported in degrees $[-b,-a]$.  By Lemma 
\ref{lem:cancellation}, this representative may be chosen
$\Phi$-minimal without enlarging its support.

\subsection{A finite-filtration ghost lemma}

Fix once and for all a weight Postnikov tower for
$M\in\C_{w\in[a,b]}$, and put $L:=b-a+1.$

\begin{definition}\label{def:strict-tower-ghost}
An endomorphism $g\colon M\to M$ is a \emph{strict tower ghost} for the fixed
tower if it is the underlying map of a morphism from the tower to itself
whose induced maps on all weight factors $H_i[i]$ are zero.
\end{definition}

The adjective ``strict'' is essential: vanishing only in the weak homotopy
category does not by itself give the following filtration argument.

\begin{lemma}
\label{lem:strict-ghost}
More generally, if $g_1,\ldots,g_r$, where $1\le r\le L$, are strict tower
ghosts, then $g_r\cdots g_1$ factors through the structural map
$M_{\le b-r}\to M$, where $M_{\le j}=0$ for $j<a$.  In particular, a
composite of $L$ strict tower ghosts is zero.
\end{lemma}

\begin{proof}
Choose tower representatives
$g_{\nu,\le i}\colon M_{\le i}\to M_{\le i}$.  Write the $i$th triangle as
\[
  M_{\le i-1}\xrightarrow{j_{i-1}}M_{\le i}
  \xrightarrow{c_i}H_i[i]\longrightarrow M_{\le i-1}[1].
\]
Since $g_\nu$ is strict, its map on $H_i[i]$ is zero, and hence
$c_i g_{\nu,\le i}=0$.  Exactness gives a factorization
\[
  g_{\nu,\le i}=j_{i-1}r_{\nu,i}
\]
for some $r_{\nu,i}\colon M_{\le i}\to M_{\le i-1}$.

Let $J_{s,i}\colon M_{\le s}\to M_{\le i}$ denote the structural map.  For
each $i\in[a,b]$ and $1\le s\le\min\{r,i-a+1\}$, we prove by induction on
$s$ that
$g_{s,\le i}\cdots g_{1,\le i}$ factors through $J_{i-s,i}$.  The case
$s=1$ is the preceding factorization.  If the assertion holds for $s$,
compatibility of $g_{s+1}$ with the tower gives
\[
  g_{s+1,\le i}J_{i-s,i}
  =J_{i-s,i}g_{s+1,\le i-s},
\]
and the last map factors through $M_{\le i-s-1}$.  This proves the induction
step.  Taking $i=b$ and $s=r$ proves the asserted factorization through
$M_{\le b-r}$.  If $r=L$, this object is $M_{\le a-1}=0$.
\end{proof}

\subsection{The kernel of the strong weight complex}

Set
\[
  K_M:=\ker\!\left(
  \End_{\C}(M)\xrightarrow{\ t\ }
  \End_{K^b(\Hh)}(t(M))
  \right).
\]

\begin{lemma}
\label{lem:strictification}
Fix a weight Postnikov tower for $M\in\C_{w\in[a,b]}$.
\begin{enumerate}[label=\textup{(\roman*)}]
\item Every $h\in K_M$ extends to an endomorphism of the fixed tower whose
induced chain map on the weight factors is weakly null-homotopic.
\item If $h,h'\in K_M$, then the composite $h'h$ can be represented by an
endomorphism of the fixed tower that is a strict tower ghost.
\item Given $h\in K_M$ and a weight factor of the fixed tower, $h$ can be
represented by an endomorphism of that tower whose component on the chosen
factor is zero.
\end{enumerate}
\end{lemma}

\begin{proof}
By \cite[Corollary~3.5]{Sosnilo}, the composite of the strong weight-complex
functor with the projection
\[
  K^b(\Hh)\longrightarrow K^b_{\mathfrak w}(\Hh)
\]
where the target is weak homotopy category, is naturally
isomorphic to Bondarko's weak weight-complex functor.  By
\cite[Proposition~1.3.4(2)]{BondarkoWeightComplexes}, every $h$ extends to
the fixed tower.  Hence, if $h\in K_M$, the chain map induced by any such
extension is weakly homotopic to zero.  This proves (i).

For (ii), choose tower extensions of $h$ and $h'$.  Their chain maps are
weakly null-homotopic by (i), so their composite is zero in the ordinary
homotopy category by
\cite[Proposition~B.2(4)]{BondarkoWeightComplexes}.  Proposition~1.3.4(13)
of the same source realizes the zero chain map by an endomorphism of the same
tower without changing the underlying map $h'h$.  This is a strict tower
ghost.

For (iii), \cite[Proposition~B.2(1),(2)]{BondarkoWeightComplexes} replaces
the weakly null chain map by an ordinarily homotopic one whose component in
the prescribed degree is zero.  Again Proposition~1.3.4(13) realizes this
chain map on the fixed tower without changing the underlying endomorphism.
\end{proof}

\begin{lemma}
\label{lem:strong-kernel}
If $M\in\C_{w\in[a,b]}$ and $L=b-a+1$, then
\[
  K_M^{2L-1}=0.
\]
\end{lemma}

\begin{proof}
Choose
\[
  h_1,\ldots,h_{2L-1}\in K_M.
\]
If $L=1$, apply Lemma \ref{lem:strictification}(iii) to the unique weight factor.
The resulting tower has only that factor, so its underlying map $h_1$ is
zero.

Assume $L\ge2$.  For $j=1,\ldots,L-1$, put $g_j:=h_{2j}h_{2j-1}.$
By Lemma \ref{lem:strictification}(ii), every $g_j$ can be represented as a strict
tower ghost for the fixed tower.

Put $G=g_{L-1}\cdots g_1$.  By the  Lemma 
\ref{lem:strict-ghost}, the composite $G$
factors through the structural map $i\colon M_{\le a}\to M$; write
$G=i\circ u$.  Since $M_{\le a-1}=0$, the object $M_{\le a}$ identifies
with the first weight factor $H_a[a]$: the map
$M_{\le a}\to H_a[a]$ in the initial tower triangle is an isomorphism.  Apply Lemma 
\ref{lem:strictification}(iii) to $h_{2L-1}$ and this factor.  The resulting
tower representative has zero component on $H_a[a]$ and therefore zero
restriction to $M_{\le a}$, so
$h_{2L-1}i=0$.  Therefore
\[
  h_{2L-1}\cdots h_1=h_{2L-1}G=h_{2L-1}iu=0.
\]
Thus $K_M^{2L-1}=0$.
\end{proof}

\begin{proof}[Proof of \cref{thm:nilpotence}]
The case $M=0$ is immediate, so assume $M\ne0$.  Choose a weight range
$[a,b]=[a(M),b(M)]$, which is valid by \cref{thm:weight-profile}, and put
$L=L(M)$.  Represent the strong weight complex
$t(M)$ by a $\Phi$-minimal bounded complex $C^\bullet$ supported in
$[-b,-a]$.  By Proposition \ref{prop:no-exotic-retracts}, every term is a finite direct
sum of diagonal spheres.  By \cref{thm:weight-profile},
$d(C^\bullet)=d(M)$.

Let $f\in I_M$.  By definition, $F(f)=0$ in $\D$.  Naturality
\eqref{eq:weight-complex-base-change} gives
\[
  K^b(\Phi)(t(f))=t_E(F(f))=0.
\]
After transporting $t(f)$ to the chosen representative $C^\bullet$, its
class lies in the ideal $R_{C^\bullet}$ of Lemma 
\ref{lem:minimal-complex-kernel}.  That lemma shows that the strong weight
complex kills every product of $d(M)$ elements of $I_M$.  In other words,
$I_M^{d(M)}\subseteq K_M.$
By Lemma \ref{lem:strong-kernel}, $K_M^{2L-1}=0$.  Hence
\[
  I_M^{\,d(M)(2L(M)-1)}
  =(I_M^{d(M)})^{2L(M)-1}=0.
\]
This proves nilpotence of the entire two-sided ideal $I_M$.

Finally, by full faithfulness of $\pi_{**}$ on cellular $E$-modules
(\cref{thm:tanania-input}), the condition $E_{**}(f)=0$ is equivalent to
$F(f)=0$ as a morphism in $\D$.  Thus every degree-zero self-map with zero
$E$-homology belongs to $I_M$ and is nilpotent.
\end{proof}

\begin{proof}[Proof of Corollary \ref{cor:nilpotence-detection}]
Functoriality shows that a nilpotent endomorphism has nilpotent
$E$-homology.  Conversely, if $E_{**}(f)^m=0$, then $f^m\in I_M$.
The bound in \cref{thm:nilpotence} gives
\[
  f^{mN(M)}=(f^m)^{N(M)}=0.
\]
For a two-sided ideal $\mathcal J$ as in (ii), one has
$\mathcal J^m\subseteq I_M$ and hence
$\mathcal J^{mN(M)}=(\mathcal J^m)^{N(M)}=0$.

For (iii), polynomial evaluation is defined by the canonical $\F$-linearity
of Corollary \ref{cor:F2-linear}.  Full faithfulness of $\pi_{**}$ on cellular $E$-modules
again shows that $P(E_{**}(f))=0$ is equivalent to
$F(P(f))=0$.  Thus $P(f)\in I_M$, and \cref{thm:nilpotence} gives
$P(f)^{N(M)}=0$.  Apply this to $P=\mu_{\bar f}$.  The minimal polynomial of
an endomorphism of the finite-dimensional vector space $E_{**}(M)$ has degree
at most its total dimension, which gives the stated bound.
\end{proof}

\begin{proof}[Proof of Corollary  \ref{thm:fitting}]
The assertion is immediate for $M=0$, so assume $M\ne0$.  By Corollary
\ref{cor:F2-linear}, polynomial expressions in $f$ with coefficients in
$\F$ are defined.  Write $\bar f=E_{**}(f)$ and factor its minimal polynomial
as
\[
  \mu_{\bar f}(x)=x^r g(x),
  \quad g(0)\ne0.
\]
Put $N=N(M)$.  By Corollary \ref{cor:nilpotence-detection}(iii),
\[
  0=\mu_{\bar f}(f)^N=f^{rN}g(f)^N.
\]
The polynomials $x^{rN}$ and $g(x)^N$ are coprime.  Choose
$\alpha,\beta\in\F[x]$ with
\[
  \alpha(x)x^{rN}+\beta(x)g(x)^N=1
\]
and set $e=\beta(f)g(f)^N$.  Then $e-1=-\alpha(f)f^{rN}$, so both
$e^2-e$ and $f^{rN}e$ are multiples of $f^{rN}g(f)^N$.  Hence
\[
  e^2=e,
  \quad
  f^{rN}e=0,
  \quad
  g(f)^N(1-e)=0.
\]
If $r=0$, these identities give $e=0$ and
$M_{\mathrm{nil}}=0$; if $g=1$, they give $e=1$ and
$M_{\mathrm{inv}}=0$.  
Since $\C$ is idempotent-complete, $e$ gives an $f$-stable decomposition
\[
  M\cong M_{\mathrm{nil}}\oplus M_{\mathrm{inv}},
  \quad
  M_{\mathrm{nil}}:=\operatorname{im}(e),
  \quad
  M_{\mathrm{inv}}:=\operatorname{im}(1-e).
\]
The first displayed identity shows that $f$ is nilpotent on
$M_{\mathrm{nil}}$.  Since $x$ and $g(x)^N$ are coprime, there exist
$s,t\in\F[x]$ with $s(x)x+t(x)g(x)^N=1$.  On
$M_{\mathrm{inv}}$ the last term vanishes, so $s(f)$ is an inverse to $f$.
Applying $E_{**}$ to $e$ gives the usual polynomial projector onto the
generalized zero-eigenspace of $\bar f$, hence the ordinary Fitting
decomposition.

It remains to prove canonicity and functoriality.  Let
$u\colon(M,f)\to(M',f')$ satisfy $uf=f'u$, and choose Fitting decompositions
on source and target.  Any component from the nilpotent summand of $M$ to the
invertible summand of $M'$ is zero: if $f^a=0$ on the source, then
$(f')^a u=uf^a=0$, and $(f')^a$ is invertible on the target.  Similarly, a
component from the invertible summand of $M$ to the nilpotent summand of $M'$
is zero.  If $e$ and $e'$ are the respective nilpotent projectors, the two
vanishing statements are
\[
  (1-e')ue=0,
  \quad
  e'u(1-e)=0.
\]
They imply $ue=e'ue=e'u$, which is the asserted functoriality.  For two
Fitting projectors $e,e'$ of the same pair $(M,f)$, take $u=1_M$ to obtain
$e=e'e=e'$.  This proves uniqueness.
\end{proof}

\begin{corollary}\label{cor:krull-schmidt}
The category $\C$ is Krull--Schmidt.  
For an indecomposable object $M$, every endomorphism is either invertible or
nilpotent, and $\End_{\C}(M)$ is a local ring.
\end{corollary}

\begin{proof}
For a nonzero object put $D(M)=\dim_{\F}E_{**}(M)$.  This is a positive
finite integer by Theorem  \ref{thm:pure-nakayama} and Lemma \ref{lem:compact-E-modules}.  In a
nontrivial decomposition $M\cong M_1\oplus M_2$, both $D(M_i)$ are positive
and smaller than $D(M)$.  Repeated decomposition therefore terminates and
gives a finite sum of indecomposable objects.

If $M$ is indecomposable, the two summands in \cref{thm:fitting} cannot both
be nonzero.  Hence every endomorphism is either invertible or nilpotent.  If
$f$ is not invertible, then $f$ is nilpotent and $1-f$ is invertible.  The standard local-ring criterion now shows that
$\End_{\C}(M)$ is local.  The usual Krull--Schmidt argument for finite sums
of objects with local endomorphism rings proves uniqueness.
\end{proof}

\begin{remark}\label{rem:nilpotence-factor-two}

1.) As another immediate consequence, an idempotent endomorphism in $\C$ is zero
if and only if its action on isotropic $\MBP$-homology is zero: such an
idempotent lies in the nilpotent ideal $I_M$, and a nilpotent idempotent
vanishes.

2.) The exponent \(2L-1\) reflects the passage from the strong weight complex
to a strict endomorphism of a fixed weight tower. An element of \(K_M\)
induces only a weakly null-homotopic map on such a tower. By
\cite[Proposition~B.2(4)]{BondarkoWeightComplexes}, the composite of two
such maps is ordinarily null-homotopic, and
\cite[Proposition~1.3.4(13)]{BondarkoWeightComplexes} realizes the zero
chain map strictly. Thus the first \(2L-2\) factors can be paired into
\(L-1\) strict tower ghosts; their composite factors through the first weight
factor, which the remaining map kills after the endpoint adjustment of
\cite[Proposition~B.2(1),(2)]{BondarkoWeightComplexes}. Hence the distinction
between strong, ordinary, and weak homotopy is essential.
\end{remark}

\end{document}